\documentclass[reqno]{amsart}
\usepackage[T1]{fontenc}
\usepackage{lmodern}
\usepackage{amsmath,amssymb,amsthm}
\usepackage[colorlinks=true,linkcolor=blue,citecolor=blue,urlcolor=blue]{hyperref}
\numberwithin{equation}{section}
\newtheorem{theorem}{Theorem}[section]
\newtheorem{corollary}[theorem]{Corollary}

\newtheorem{lemma}[theorem]{Lemma}
\theoremstyle{remark}
\newtheorem{remark}[theorem]{Remark}
\newcommand{\HH}{\mathbb H}
\newcommand{\RR}{\mathbb R}

\newcommand{\supp}{\operatorname{supp}}
\newcommand{\Hull}{\operatorname{Hull}}
\newcommand{\Isom}{\operatorname{Isom}}
\newcommand{\Tr}{\operatorname{Tr}}
\newcommand{\Ker}{\operatorname{ker}}
\title{Singularity of harmonic measures for hyperbolic lattices}
\author{Nikolay Bogachev}
\address{Department of Computer and Mathematical Sciences, University of Toronto Scarborough, 1095 Military Trail, Toronto, ON M1C 1A3, Canada}
\email{n.bogachev@utoronto.ca}
\date{}

\begin{document}

\begin{abstract}
Let $\Gamma<\Isom(\HH^n)$, $n\geq2$, be a cocompact lattice containing a convex cocompact codimension-one subgroup in the sense of Sageev. We prove that every finitely supported admissible random walk on $\Gamma$ has hitting measure singular with respect to Lebesgue measure on $\partial\HH^n$. The proof extends the method of Kosenko--Tiozzo, replacing Fourier analysis for a cyclic subgroup by a von Neumann dimension argument that also applies to nonabelian subgroups. As corollaries, we prove the singularity conjecture of Kaimanovich and Le Prince for cocompact hyperbolic lattices admitting a proper cocompact cubulation, cocompact Kleinian groups, and cocompact hyperbolic lattices with totally geodesic sublattices of codimension one. In particular, our result covers all cocompact hyperbolic reflection groups, cocompact arithmetic lattices of simplest type and cocompact nonarithmetic lattices arising from the hybrid and inbreeding constructions. 
\end{abstract}

\maketitle

\section{Introduction}

Let $\Gamma<\Isom(\HH^n)$ be a lattice, where $n\geq2$, and let $\mu$ be a finitely supported probability measure on $\Gamma$ which is also {\em admissible}, meaning that its support generates $\Gamma$ as a semigroup. The associated right random walk converges almost surely to the sphere at infinity $\partial\HH^n$. Its {\em hitting measure} $\nu$, also called the {\em harmonic measure}, is the unique $\mu$-stationary probability measure on $\partial\HH^n$; see Kaimanovich \cite{Kai00}. Let $\lambda$ be the visual probability measure based at a point $o\in\HH^n$; its measure class does not depend on $o$.

The comparison between harmonic measure and Lebesgue measure goes back to Furstenberg \cite{Fur71}, who constructed infinitely supported random walks on lattices with absolutely continuous hitting measures. For finitely supported random walks, the singularity conjecture of Kaimanovich--Le Prince \cite{KLP11} predicts that $\nu\perp\lambda$. We prove this whenever a cocompact lattice $\Gamma < \Isom(\HH^n)$ contains a convex cocompact {\em codimension-one subgroup} $H < \Gamma$ in the sense of Sageev \cite{Sageev95}, which means that the Schreier graph $H\backslash\operatorname{Cay}(\Gamma,S)$ has more than one end, where $S$ is any finite generating set of $\Gamma$.

\begin{theorem}\label{thm:main}
Let $n\geq2$ and let $\Gamma<\Isom(\HH^n)$ be a cocompact lattice. Suppose that $\Gamma$ contains a convex cocompact codimension-one subgroup $H$. Then every finitely supported admissible probability measure $\mu$ on $\Gamma$ has hitting measure singular with respect to Lebesgue measure.
\end{theorem}

For subgroups of a cocompact hyperbolic lattice, quasiconvexity and convex cocompactness are equivalent, as follows from Bowditch \cite[Proposition~2.5.4]{Bow95}. For such a subgroup $H$, the codimension-one condition is equivalent to the existence of a partition $\partial\HH^n\setminus\Lambda_H=\Omega_-\sqcup\Omega_+$ into two nonempty open $H$-invariant sets, where $\Lambda_H$ denotes the limit set of $H$; see Fioravanti--Hagen \cite[Definition~2.19 and Remark~2.21]{FH21}.

For instance, the hypothesis holds whenever $\Lambda_H$ is a topologically embedded $(n-2)$-sphere: by the Jordan–Brouwer theorem its complement has two components, preserved by a subgroup of index at most two. The limit set need not be a sphere, however. A single such subgroup suffices; in particular, we do not assume that $\Gamma$ is cubulated. The sets $\Omega_\pm$ need not be connected, the lattice $\Gamma$ may have torsion, and the measure $\mu$ need not be symmetric.

For nonuniform Fuchsian lattices, the singularity conjecture follows from the work of Guivarc'h--Le Jan \cite{GLJ93}; see also Blach\`ere--Ha\"issinsky--Mathieu \cite[Theorem~1.10]{BHM11}. In all dimensions, the nonuniform case is covered by Gekhtman--Tiozzo \cite[Corollary~4.2]{GT} and Randecker--Tiozzo \cite{RT21}. For cocompact Fuchsian groups, Kosenko--Tiozzo \cite{KT22} proved the singularity conjecture for measures supported on the side-pairing generators of centrally symmetric fundamental polygons. The general case in dimension two was independently proved by B\'enard \cite{Benard26} and Kosenko--Tiozzo \cite{KT}. For $n=2$, a cyclic subgroup generated by an orientation-preserving loxodromic element is convex cocompact and codimension one, so Theorem~\ref{thm:main} recovers these results.

Kahn--Markovic \cite[Theorem~1.1]{KM} constructed quasi-Fuchsian surface subgroups in every closed hyperbolic $3$-manifold group. Their limit sets are topological circles, so Theorem~\ref{thm:main} gives the following corollary, after passing to a torsion-free finite-index subgroup.

\begin{corollary}\label{cor:kleinian}
Let $\Gamma<\Isom(\HH^3)$ be a cocompact lattice. Every finitely supported admissible probability measure $\mu$ on $\Gamma$ has hitting measure singular with respect to Lebesgue measure.
\end{corollary}

The Kahn--Markovic surface subgroups were also used by Bergeron--Wise \cite[Theorem~5.3]{BW12} to construct proper cocompact actions of closed hyperbolic $3$-manifold groups on $\mathrm{CAT}(0)$ cube complexes. This was a key ingredient in Agol's proof of the virtual Haken and virtual fibering conjectures \cite{Agol13}, building on Wise's work on quasiconvex hierarchies; see \cite{Wise21}.  

In recent joint work with Kosenko and Tiozzo \cite{BKT25}, we used algebraic and geometric convergence to study singularity on families of Fuchsian and Kleinian groups. In particular, hyperbolic Dehn filling gives singularity for suitable measures on cocompact Kleinian groups approaching a nonuniform lattice. Corollary~\ref{cor:kleinian} applies to every cocompact Kleinian group and every finitely supported admissible measure.

The next corollary covers a broad class of cocompact hyperbolic lattices in all dimensions, as we explain below.

\begin{corollary}\label{cor:geodesic}
Let $\Gamma<\Isom(\HH^n)$, $n\geq2$, be a cocompact hyperbolic lattice such that the orbifold $\HH^n/\Gamma$ admits a properly immersed, totally geodesic subspace of codimension one. Then every finitely supported admissible probability measure $\mu$ on $\Gamma$ has hitting measure singular with respect to Lebesgue measure.
\end{corollary}

If $\Gamma$ is a cocompact arithmetic lattice of simplest type, associated with a quadratic form $f$ over a totally real field $k$, then every $k$-rational hyperplane on which $f$ has signature $(n-1,1)$ yields a closed immersed totally geodesic hypersurface in $\HH^n/\Gamma$; see, for example, Belolipetsky et al.\ \cite{BBKS26} for a general account. Thus Corollary~\ref{cor:geodesic} applies. Since every arithmetic lattice in even dimensions is of simplest type, this covers all cocompact arithmetic lattices in even dimensions. Such hypersurface subgroups were also used by Bergeron--Haglund--Wise \cite{BHW11} to cubulate cocompact arithmetic lattices of simplest type.

Our method also applies to cocompact nonarithmetic orbifolds containing a closed totally geodesic hypersurface. These include the hybrids of Gromov--Piatetski-Shapiro \cite{GPS87}, Agol's short-systole construction \cite{Agol06} and its higher-dimensional extensions, obtained independently by Belolipetsky--Thomson \cite{BT11} and Bergeron--Haglund--Wise \cite[Section~9]{BHW11}, and the hybrids of Douba \cite{Douba25}. These examples include both quasi-arithmetic and non-quasi-arithmetic lattices; see Thomson \cite{Thomson16} and Douba \cite{Douba25}.

Cocompact hyperbolic reflection groups are also covered by Corollary~\ref{cor:geodesic}. Indeed, let $r\in W$ be a reflection in a cocompact reflection group $W<\Isom(\HH^n)$, with fixed hyperplane $\Pi$. Its centralizer $C_W(r)$ acts cocompactly on $\Pi$; see Raghunathan \cite[Theorem~1.13 and Lemma~1.14]{Rag72}. 

A proper cocompact cubulation also gives a quasiconvex codimension-one subgroup, by taking the side-preserving stabilizer of an essential hyperplane. We therefore obtain the following consequence.

\begin{corollary}\label{cor:cubulation}
Let $\Gamma<\Isom(\HH^n)$, $n\geq2$, be a cocompact lattice admitting a proper cocompact action on a $\mathrm{CAT}(0)$ cube complex. Then every finitely supported admissible probability measure $\mu$ on $\Gamma$ has hitting measure singular with respect to Lebesgue measure.
\end{corollary}

Wise \cite[Conjecture~13.52]{Wise14} conjectured that every cocompact real hyperbolic lattice admits such an action. This remains open in higher dimensions, in particular for arithmetic lattices of the second type in odd dimensions at least five and the exceptional trialitarian arithmetic lattices in dimension seven; see Bregman--Incerti-Medici \cite[Introduction]{BIM26}. Even the existence of one quasiconvex codimension-one subgroup is unknown in general: for a cocompact arithmetic lattice, one such subgroup would already imply cubulation by Bergeron--Wise \cite[Theorem~4.1]{BW12}. Our argument shows that a positive answer to this existence question would prove the singularity conjecture for all cocompact real hyperbolic lattices. Together with the known nonuniform case, this would settle the conjecture for all real hyperbolic lattices. In particular, Wise's cubulation conjecture would imply the singularity conjecture.

\subsection*{Line of argument and organization}
Our proof extends the method of Kosenko--Tiozzo \cite{KT}. They express the Green kernel in terms of its restrictions to a neighborhood of a loxodromic axis and show that the associated Na\"im current is proportional to the Liouville current under the nonsingularity assumption. We replace the cyclic subgroup by $H$ and the axis by the convex hull of $\Lambda_H$.

The key new ingredient is a von Neumann dimension argument instead of the Fourier analysis used for the cyclic subgroup in \cite[Lemma~5]{KT}. Finite support and convex cocompactness provide an $H$-invariant set $A\subset\Gamma$ which every random-walk path between the two sides must visit and which consists of finitely many $H$-orbits. Thus $\ell^2(A)$ is a finite direct sum of copies of $\ell^2(H)$. Using the corresponding boundary-kernel identity, we construct a bounded $H$-equivariant operator from an $L^2$-space on a suitable invariant subset of one side to $\ell^2(A)$. Its source has infinite von Neumann $H$-dimension, whereas its target has finite von Neumann $H$-dimension, so the operator has a nonzero kernel. On the other hand, proportionality of the Na\"im and Liouville currents would force this operator to be injective. The latter follows from the injectivity of the relevant Liouville integral transform, which we prove using polynomial moments on a bounded domain. 

Section~\ref{sec:boundary} recalls the boundary facts and constructs the set $A$. Section~\ref{sec:green} establishes the identity expressing the Green kernel in terms of its restrictions to $A$ and passes to the boundary limit. Section~\ref{sec:current} identifies the Na\"im current and constructs a bounded $H$-equivariant operator with nonzero kernel by comparing von Neumann dimensions. The injectivity lemma in Section~\ref{sec:injectivity} gives a contradiction, completing the proof in Section~\ref{sec:proofs}.

\subsection*{Funding} This work was supported by NSERC Discovery Grant RGPIN-2024-05680.

\subsection*{AI use and ackowledgements} The key ideas belong to me, but I used large language models to develop technical arguments, search the literature, and check the proofs. I thank Peter Kosenko and Giulio Tiozzo for fruitful discussions and useful comments. I am also grateful to the organizers of the thematic program ``Randomness and Geometry'' that was held at the Fields Institute in Toronto in 2024; this program introduced me to random walks on groups. 

\section{Boundary measures and convex cocompact subgroups}\label{sec:boundary}

Fix $\Gamma$ and $H$ as in Theorem~\ref{thm:main}, and put $d=n-1$. Let $\mu_0$ be the initial measure on $\Gamma$ from Theorem~\ref{thm:main}. Replace $\mu_0$ by $\mu=\tfrac12(\delta_e+\mu_0)$, where $\delta_e$ is the Dirac mass at $e \in \Gamma$. The $\mu$-walk is a random time change of the $\mu_0$-walk and therefore has the same hitting measure. For a measure $\eta$ on $\Gamma$, write $\check\eta(g)=\eta(g^{-1})$. The same argument applies to $\check\mu=\tfrac12(\delta_e+\check\mu_0)$. Denote the hitting measures of $\mu$ and $\check\mu$ by $\nu$ and $\check\nu$, respectively. The following lemma is standard; see Gekhtman--Tiozzo \cite[Proof of Proposition~4.6 and Corollary~4.8]{GT}.

\begin{lemma}\label{lem:dichotomy}
The measure $\nu$ is either singular or equivalent to $\lambda$. If $\nu$ is equivalent to $\lambda$, then $\check\nu$ is equivalent to $\lambda$ as well.
\end{lemma}

Suppose from now on, towards a contradiction, that $\nu$ is not singular. By Lemma~\ref{lem:dichotomy}, we have $\nu\sim\lambda\sim\check\nu$.

Since $H$ is convex cocompact and codimension-one, choose a partition $\partial\HH^n\setminus\Lambda_H=\Omega_-\sqcup\Omega_+$ into nonempty open $H$-invariant sets, as in Fioravanti--Hagen \cite[Definition~2.19 and Remark~2.21]{FH21}. Put $Q=\Hull(\Lambda_H)$.

By Selberg's lemma, we may replace $H$ by a torsion-free subgroup of finite index. This preserves the limit set, convex cocompactness and invariance of $\Omega_\pm$. The action of $H$ on each $\Omega_\pm$ is properly discontinuous and, since $H$ is torsion-free, free.

Put $L=\max_{s\in\supp\mu}d(o,so)$ and $A=\{a\in\Gamma:d(ao,Q)\leq L+1\}$. The set $A$ is invariant under left multiplication by $H$. Since $H$ acts cocompactly on $Q$ and $\Gamma$ acts properly on $\HH^n$, the quotient $H\backslash A$ is finite. Thus
\begin{equation}\label{eq:finite-orbits}
A=\bigsqcup_{j=1}^N Hb_j.
\end{equation}
A {\em random-walk path} is a finite sequence $g_0,\ldots,g_m\in\Gamma$ such that $g_{i-1}^{-1}g_i\in\supp\mu$ for $1\leq i\leq m$.
\begin{lemma}\label{lem:crossing}
Let $x_k,y_k\in\Gamma$ satisfy $x_ko\to\xi\in\Omega_-$ and $y_ko\to\eta\in\Omega_+$. For all sufficiently large $k$, these two points lie in different components of $\HH^n\setminus Q$, and every random-walk path from $x_k$ to $y_k$ visits $A$.
\end{lemma}

\begin{proof}
Use the Klein ball model and choose $p\in Q$. For $x\notin Q$, let $r_p(x)$ be the endpoint on the boundary sphere of the Euclidean ray from $p$ through $x$. If $r_p(x)$ belonged to $\Lambda_H$, convexity would give $[p,r_p(x))\subset Q$, contrary to $x\notin Q$. Thus radial projection defines a continuous map from $\HH^n\setminus Q$ to $\partial\HH^n\setminus\Lambda_H$.

The ideal boundary of $Q$ is $\Lambda_H$. Hence $x_k o$ and $y_k o$ eventually lie outside $Q$, and their radial projections tend to $\xi$ and $\eta$. For large $k$ the projections lie in $\Omega_-$ and $\Omega_+$, respectively. A path in $\HH^n\setminus Q$ cannot connect the two points: its radial image is connected and cannot meet both parts of this open partition.

Interpolate the orbit points of a random-walk path by geodesic segments. Each segment has length at most $L$. The interpolated path must meet $Q$, so an endpoint of one of its segments is within distance $L$ of $Q$. The corresponding group element belongs to $A$.
\end{proof}

By \eqref{eq:finite-orbits}, the accumulation set of $Ao$ at infinity is exactly $\Lambda_H$.

\section{The Green operator and its boundary factorization}\label{sec:green}

Define the Markov operator and Green kernel by
$$
(Pf)(x)=\sum_s\mu(s)f(xs),\qquad
G(x,y)=\sum_{k\geq0}\mu^{*k}(x^{-1}y).
$$
The lattice $\Gamma$ contains a nonabelian free subgroup and is therefore nonamenable. The positive self-adjoint operator $P^*P$ is convolution by the symmetric probability measure $\check\mu*\mu$. Since $\mu(e)>0$, its support contains $\supp\mu_0$ and $(\supp\mu_0)^{-1}$ and generates $\Gamma$. Kesten's criterion \cite{Kesten} gives $\|P\|^2=\|P^*P\|<1$. Consequently $G=(I-P)^{-1}=\sum_{k\geq0}P^k$ is bounded on $\ell^2(\Gamma)$ and has matrix $G(x,y)$. 

For $v\in\ell^2(\Gamma)$, put $w=Gv$, so that $v=w-Pw$. Then
$$
2\operatorname{Re}\langle Gv,v\rangle
=\|v\|^2+\|w\|^2-\|Pw\|^2\geq\|v\|^2.
$$
The compression $G_A:\ell^2(A)\to\ell^2(A)$, with matrix $(G(a,b))_{a,b\in A}$, satisfies the same inequality. It is bounded below and has closed range. Its adjoint has the same real quadratic form and is also bounded below, so the range of $G_A$ is dense. Thus $G_A$ is boundedly invertible; put $M_A=G_A^{-1}$. Both operators commute with left $H$-translations and preserve real-valued vectors. This is the coercivity argument of \cite[Lemmas~2--3]{KT}.

\begin{lemma}\label{lem:green-factor}
If every random-walk path from $x$ to $y$ visits $A$, then
\begin{equation}\label{eq:green-factor}
G(x,y)=G(x,A)M_A G(A,y).
\end{equation}
Here $G(x,A)=(G(x,a))_{a\in A}$ is a row vector and $G(A,y)=(G(a,y))_{a\in A}$ is a column vector. Both belong to $\ell^2(A)$, so the right-hand side is well defined.
\end{lemma}

\begin{proof}
Let $F_x(a)$ be the probability that the first visit to $A$ occurs at $a$, allowing a visit at time zero. Then $F_x\in\ell^1(A)\subset\ell^2(A)$. The strong Markov property gives $G(x,A)=F_xG_A$ coordinatewise. Every row and column of $G$ belongs to $\ell^2$, so this is also an equality of $\ell^2$ row vectors, and $F_x=G(x,A)M_A$. A second application of the strong Markov property gives $G(x,y)=F_x G(A,y)$.
\end{proof}

The group $\Gamma$ is word hyperbolic, with Gromov boundary $\partial\HH^n$. We use Ancona inequalities and the identification of the Martin boundary without assuming symmetry. Indeed, Gou\"ezel's theorem \cite[Theorem~1.2]{Gouezel15} applies: finite support gives superexponential tails, and the remaining hypothesis is that $\sum_{k\geq0}r^k\mu^{*k}(x^{-1}y)<\infty$ for some $r>1$ and all $x,y\in\Gamma$. This follows by choosing $1<r<\|P\|^{-1}$, for which $\sum_{k\geq0}r^kP^k$ converges in operator norm. Strong Ancona inequalities are established in \cite[Section~4.2 and Theorem~4.10]{Gouezel15}.

Our conventions for the Martin kernels are
$$
K_\mu(a,\eta)=\lim_{y\to\eta}\frac{G(a,y)}{G(e,y)}
=\frac{d(a_*\nu)}{d\nu}(\eta),
$$
and similarly for $\check\mu$. The Na\"im kernel is
$$
\Theta(\xi,\eta)=
\lim_{x\to\xi,\,y\to\eta}
\frac{G(x,y)}{G(x,e)G(e,y)},\qquad \xi\ne\eta.
$$
It is positive and continuous off the diagonal. Its existence, transformation rule and associated invariant current for finitely supported nonsymmetric walks are recorded by Cantrell--Tanaka \cite[Section~6]{CT24}. The following lemma extends Kosenko--Tiozzo \cite[Lemma~7]{KT}.

\begin{lemma}\label{lem:boundary-factor}
For $\xi\in\Omega_-$ and $\eta\in\Omega_+$, the vectors $u(\xi)=(K_{\check\mu}(a,\xi))_{a\in A}$ and $v(\eta)=(K_\mu(a,\eta))_{a\in A}$ belong to $\ell^2(A)$, and
\begin{equation}\label{eq:boundary-factor}
\Theta(\xi,\eta)=u(\xi)^t M_Av(\eta).
\end{equation}
\end{lemma}

\begin{proof}
Choose sequences $x_k\to\xi$ and $y_k\to\eta$ in the boundary compactification of $\Gamma$. By Lemma~\ref{lem:crossing}, the Green factorization holds for large $k$. Divide it by $G(x_k,e)G(e,y_k)$. The row and column coordinates converge to $u(\xi)$ and $v(\eta)$, respectively; for the row, use $G_{\check\mu}(a,x)=G(x,a)$.

We prove convergence in $\ell^2(A)$. Use a word metric on $\Gamma$. Since the accumulation set of $A$ is $\Lambda_H$ and $\xi\notin\Lambda_H$, the Gromov products $(x_k\mid a)_e$ are uniformly bounded for $a\in A$ and all sufficiently large $k$. Otherwise some sequence in $A$ would converge to $\xi$. Hence $e$ is within a fixed distance of every geodesic segment from $x_k$ to $a$. Ancona inequalities and Harnack comparison over that bounded distance give
$$
\frac{G(x_k,a)}{G(x_k,e)}\leq C_\xi G(e,a),\qquad
\frac{G(a,y_k)}{G(e,y_k)}\leq C_\eta G(a,e).
$$
The dominating vectors belong to $\ell^2(A)$, since the full row and column are $G^*\delta_e$ and $G\delta_e$, respectively. Dominated convergence gives both limits in $\ell^2(A)$. Boundedness of $M_A$ then permits passage to the limit in \eqref{eq:green-factor}, proving \eqref{eq:boundary-factor}.
\end{proof}

\section{Boundary currents and equivariant operators}\label{sec:current}

On ordered pairs of distinct boundary points, the Na\"im current $dJ(\xi,\eta)=\Theta(\xi,\eta)\,d\check\nu(\xi)\,d\nu(\eta)$ is $\Gamma$-invariant; see \cite[Section~6]{CT24}. In stereographic coordinates on $\partial\HH^n\setminus\{\infty\}\cong\RR^d$, the Liouville current is, up to normalization,
$$
d\mathcal L(z,w)=\frac{dz\,dw}{|z-w|^{2d}}.
$$
It is invariant under all isometries of $\HH^n$, including orientation-reversing ones.

Since $\nu\sim\lambda\sim\check\nu$ and $\Theta$ is positive, the currents $J$ and $\mathcal L$ are equivalent. Their Radon--Nikodym derivative is $\Gamma$-invariant. In Hopf coordinates it defines a geodesic-flow-invariant measurable function on the unit tangent bundle of the compact quotient. Hopf ergodicity makes it constant, so $J=c\mathcal L$ for some $c>0$. For an orbifold quotient, apply Hopf ergodicity on a torsion-free finite cover. This identification follows \cite[Lemma~4]{KT}.

\begin{lemma}\label{lem:reference-measure}
For $\Omega=\Omega_-$ or $\Omega_+$, the measure $m=\sum_{h\in H}h_*(\lambda|_\Omega)$ is $H$-invariant, non-atomic, $\sigma$-finite and equivalent to $\lambda|_\Omega$, with $m\geq\lambda|_\Omega$. Every Borel fundamental domain $F\subset\Omega$ satisfies $m(F)=\lambda(\Omega)<\infty$.
\end{lemma}

\begin{proof}
The free and properly discontinuous action admits a Borel fundamental domain $F$. Reindexing the sum proves invariance, and the partition of $\Omega$ by the sets $h^{-1}F$ gives $m(F)=\sum_{h\in H}\lambda(h^{-1}F)=\lambda(\Omega)$. The translates of $F$ cover $\Omega$ and have the same finite mass, proving $\sigma$-finiteness. Every summand is non-atomic and absolutely continuous with respect to $\lambda|_\Omega$, since isometries of $\HH^n$ preserve the visual measure class. The term $h=e$ gives $m\geq\lambda|_\Omega$ and the reverse absolute continuity.
\end{proof}

Denote these measures by $m_-$ and $m_+$. Choose stereographic coordinates with $\infty\in\Omega_-$, so that $\Omega_+$ has bounded closure in $\RR^d$. Write $dm_\pm(z)=p_\pm(z)\,dz$ on the finite part of each domain. The densities are positive and finite almost everywhere. The smooth visual density has a positive lower bound on $\Omega_+$. Lemma~\ref{lem:reference-measure} therefore gives $p_+\geq c_0>0$ almost everywhere, and consequently $L^2(\Omega_+,m_+)\subset L^2(\Omega_+,dw)\subset L^1(\Omega_+,dw)$.

For $a\in A$, define
$$
k_a^+(w)=\frac{d(a_*\nu)|_{\Omega_+}}{dm_+}(w),\qquad
k_a^-(z)=\frac{d(a_*\check\nu)|_{\Omega_-}}{dm_-}(z).
$$
These functions are nonnegative and their integrals are at most one. Invariance of $m_\pm$ gives
\begin{equation}\label{eq:k-equiv}
k_{ha}^\pm(z)=k_a^\pm(h^{-1}z)\qquad(h\in H).
\end{equation}
All the countably many almost-everywhere identities can be imposed on a common invariant conull set.

Let $r_+=d\nu/dm_+$ and $r_-=d\check\nu/dm_-$ on their respective domains. They are positive and finite almost everywhere. The Martin identities give $k^+(w)=r_+(w)v(w)$ and $k^-(z)=r_-(z)u(z)$, where $k^\pm$ denotes the vector indexed by $A$. By Lemma~\ref{lem:boundary-factor}, these vectors belong to $\ell^2(A)$ almost everywhere. Multiplying \eqref{eq:boundary-factor} by the hitting densities and using $J=c\mathcal L$ gives
\begin{equation}\label{eq:current-factor}
k^-(z)^tM_Ak^+(w)
=\frac{c}{p_-(z)p_+(w)|z-w|^{2d}}
\end{equation}
for almost every $(z,w)\in\Omega_-\times\Omega_+$.

\subsection*{A bounded equivariant operator}

Following Kosenko--Tiozzo \cite[Lemma~8]{KT}, put $S(w)=\sum_{a\in A}k_a^+(w)$. This function is $H$-invariant by \eqref{eq:k-equiv}. Choose a Borel fundamental domain $F\subset\Omega_+$.

Tonelli's theorem, \eqref{eq:k-equiv}, and \eqref{eq:finite-orbits} give
$$
\begin{aligned}
\int_F S\,dm_+
&=\sum_{j=1}^N\sum_{h\in H}\int_F k_{hb_j}^+\,dm_+ =\sum_{j=1}^N\sum_{h\in H}\int_{h^{-1}F}k_{b_j}^+\,dm_+\\
&=\sum_{j=1}^N\nu(b_j^{-1}\Omega_+)\leq N.
\end{aligned}
$$
Thus $S$ is finite almost everywhere. Since $m_+(F)=\lambda(\Omega_+)>0$, there is $M<\infty$ such that $F_M=F\cap\{S\leq M\}$ has positive measure. The set $E=\bigcup_{h\in H}hF_M$ is invariant and satisfies $S\leq M$ almost everywhere.

Define $R:L^2(E,m_+)\to\ell^2(A)$ by
$$
(Rf)(a)=\int_E f(w)k_a^+(w)\,dm_+(w).
$$
Each $k_a^+|_E$ is square-integrable because $\int_E(k_a^+)^2\,dm_+\leq M\int_E k_a^+\,dm_+\leq M$. Weighted Cauchy--Schwarz and Tonelli's theorem give
$$
\begin{aligned}
\|Rf\|_2^2
&\leq\sum_{a\in A}
\left(\int_E|f|^2k_a^+\,dm_+\right) \left(\int_Ek_a^+\,dm_+\right)\\
&\leq\int_E|f|^2S\,dm_+ \leq M\|f\|_{L^2(E,m_+)}^2,
\end{aligned}
$$
so $R$ is bounded.

The actions $(U_hf)(w)=f(h^{-1}w)$ and $(\lambda_hq)(a)=q(h^{-1}a)$ are unitary. A change of variables using \eqref{eq:k-equiv} gives $RU_h=\lambda_hR$. The adjoint is $R^*q=\sum_{a\in A}q(a)k_a^+|_E$, with convergence in $L^2(E,m_+)$.

\subsection*{The von Neumann trace argument}

The von Neumann dimension of a closed invariant subspace of a Hilbert direct sum of regular representations is the canonical trace of its orthogonal projection, normalized by $\dim_H\ell^2(H)=1$; see L\"uck \cite{Luck02}.

\begin{lemma}\label{lem:finite-trace}
For every discrete group $H$ and every integer $N\geq1$, a bounded $H$-equivariant operator $T:\ell^2(H)^{N+1}\to\ell^2(H)^N$ has a nonzero kernel.
\end{lemma}

\begin{proof}
Suppose that $T$ is injective. Write its polar decomposition as $T=V(T^*T)^{1/2}$, where $(T^*T)^{1/2}$ is the positive square root of $T^*T$ and the operator $V$ is isometric on $(\Ker T)^\perp$, hence on the whole source. It is $H$-equivariant by uniqueness of the polar decomposition.

Write $V=(v_{ij})$ as an $N\times(N+1)$ matrix. Its entries belong to the commutant of the left regular representation of $H$, that is, the algebra of bounded operators on $\ell^2(H)$ commuting with all left translations. The canonical trace on this algebra is $\tau(B)=(B\delta_e)(e)$. For every such $B$, equivariance gives $(B^*\delta_e)(h)=\overline{(B\delta_e)(h^{-1})}$, and hence
$$
\tau(B^*B)=\|B\delta_e\|_2^2=\|B^*\delta_e\|_2^2=\tau(BB^*).
$$
The functional $\tau$ is positive and satisfies $\tau(I)=1$. On square matrices over this algebra, let $\Tr$ be the sum of $\tau$ on the diagonal. Applying the preceding identity to each $v_{ij}$ gives
$$
\Tr(V^*V)=\sum_{i,j}\tau(v_{ij}^*v_{ij})
=\sum_{i,j}\tau(v_{ij}v_{ij}^*)=\Tr(VV^*).
$$
But $V^*V=I_{N+1}$ and $VV^*\leq I_N$, giving $N+1\leq N$, a contradiction.
\end{proof}

We have a unitary $H$-equivariant isomorphism $L^2(E,m_+)\cong\ell^2(H)\otimes L^2(F_M,m_+)$. Explicitly, $f$ corresponds to the family $f_h(w)=f(hw)$ for $w\in F_M$, and $U_g$ replaces $h$ with $g^{-1}h$.

The measure on $F_M$ is non-atomic, finite and positive. Hence $L^2(F_M,m_+)$ is infinite-dimensional, so $\dim_H L^2(E,m_+)=\infty$, whereas $\dim_H\ell^2(A)=N$. Indeed, \eqref{eq:finite-orbits} gives an $H$-equivariant isomorphism $\ell^2(A)\cong\ell^2(H)^N$. Choose $N+1$ orthonormal functions in $L^2(F_M,m_+)$. Their $H$-translates span a closed invariant subspace isomorphic to $\ell^2(H)^{N+1}$. Applying Lemma~\ref{lem:finite-trace} to the restriction of $R$ gives a nonzero $f\in L^2(E,m_+)$ with $Rf=0$. Since $R$ commutes with complex conjugation, we may assume that $f$ is real-valued. Extend it by zero to $\Omega_+$. The lower bound on $p_+$ implies that this extension is in $L^1(\Omega_+,dw)$.

\section{Injectivity of the Liouville integral transform}\label{sec:injectivity}

\begin{lemma}\label{lem:injectivity}
Let $D\subset\RR^d$, $d\geq1$, be bounded and measurable, and let $f\in L^1(D)$. If
$$
\int_D\frac{f(y)}{|x-y|^{2d}}\,dy=0
$$
for almost every $x$ outside a sufficiently large ball, then $f=0$ almost everywhere.
\end{lemma}

\begin{proof}
Enlarge a ball containing $D$ so that the assumed vanishing holds outside it. The integral is continuous there, hence vanishes everywhere outside the ball. Substituting $x=t/|t|^2$ and dividing by $|t|^{2d}$, we get
$$
\int_D f(y)F_d(t,y)\,dy=0,
\qquad
F_\alpha(t,y)=(1-2t\cdot y+|t|^2|y|^2)^{-\alpha},
$$
for all sufficiently small nonzero $t$. Since $D$ is bounded, the integral is real-analytic near $t=0$, and its derivatives can be computed under the integral.

Write $(a)_j=a(a+1)\cdots(a+j-1)$, with $(a)_0=1$. Direct differentiation gives $\Delta_tF_\alpha=4\alpha(\alpha+1-d/2)|y|^2F_{\alpha+1}$, and iteration yields
$$
\Delta_t^jF_\alpha(t,y)
=4^j(\alpha)_j(\alpha+1-d/2)_j
|y|^{2j}F_{\alpha+j}(t,y).
$$
For a harmonic homogeneous polynomial $H_l$ of degree $l$ on $\RR^d$, let $H_l(\partial_t)$ denote the differential operator obtained by replacing each variable $t_i$ by $\partial/\partial t_i$. Then $[H_l(\partial_t)F_\beta(t,y)]_{t=0}=2^l(\beta)_lH_l(y)$. Indeed, the homogeneous Taylor term of degree $l$ is $2^l(\beta)_l(t\cdot y)^l/l!+|t|^2Q(t,y)$ for a polynomial $Q$. Applying $H_l(\partial_t)$ at zero kills the second term, since $[H_l(\partial_t)(|t|^2Q)]_{t=0}=[(\Delta H_l)(\partial_t)Q]_{t=0}=0$. The first term gives the stated identity.

Applying $H_l(\partial_t)\Delta_t^j$ to the vanishing integral at zero therefore gives
$$
4^j(d)_j(d/2+1)_j\,2^l(d+j)_l
\int_D f(y)|y|^{2j}H_l(y)\,dy=0.
$$
The coefficient is nonzero, so $f(y)\,dy$ annihilates every polynomial of the form $|y|^{2j}H_l(y)$.

Every polynomial is a sum of such terms by harmonic decomposition. Consequently $f(y)\,dy$ annihilates all polynomials. As a finite complex measure on a compact ball containing $D$, it annihilates all continuous functions by Stone--Weierstrass. The measure is therefore zero.
\end{proof}

\begin{remark}
When $d=2$, one can use complex moments. For bounded $w$ and sufficiently large $z$,
$$
\frac1{|z-w|^4}
=\frac1{|z|^4}\sum_{j,k\geq0}(j+1)(k+1)
\frac{w^j\overline w^{\,k}}{z^j\overline z^{\,k}}.
$$
Vanishing of the transform forces every mixed moment $\int f(w)w^j\overline w^{\,k}\,dw$ to vanish.
\end{remark}

\section{Proofs of the main results}\label{sec:proofs}

\begin{proof}[Proof of Theorem~\ref{thm:main}]
Continue under the assumption $\nu\sim\lambda\sim\check\nu$, and take the nonzero real function $f\in\Ker R$ constructed in Section~\ref{sec:current}.

For almost every $z\in\Omega_-\setminus\{\infty\}$, we have $k^-(z)\in\ell^2(A)$. Put $q(z)=M_A^*k^-(z)\in\ell^2(A)$; all entries are real. The series defining $R^*q(z)$ converges in $L^2(E,m_+)$, so its finite partial sums have a subsequence converging almost everywhere. Since $k^+(w)\in\ell^2(A)$ for almost every $w$, these partial sums also converge pointwise to the pairing of $q(z)$ with $k^+(w)$. The two limits agree almost everywhere.

Fubini's theorem applied to \eqref{eq:current-factor} now gives, for almost every $z\in\Omega_-\setminus\{\infty\}$,
$$
R^*q(z)(w)=\frac{c}{p_-(z)p_+(w)|z-w|^{2d}}\quad\text{in }L^2(E,m_+).
$$

Taking the inner product with $f$ and using its extension by zero gives
$$
0=\langle Rf,q(z)\rangle
=\langle f,R^*q(z)\rangle
=\frac{c}{p_-(z)}
\int_{\Omega_+}\frac{f(w)}{|z-w|^{2d}}\,dw
$$
for almost every $z\in\Omega_-$. This also holds Lebesgue almost everywhere, since $m_-$ and Lebesgue measure are equivalent. The domain $\Omega_-$ contains a neighborhood of infinity, and the zero extension of $f$ is in $L^1(\Omega_+,dw)$. Lemma~\ref{lem:injectivity} gives $f=0$ Lebesgue almost everywhere, and equivalence of $m_+$ and Lebesgue measure contradicts the choice of $f$. Thus $\nu\perp\lambda$.
\end{proof}

\begin{proof}[Proof of Corollary~\ref{cor:kleinian}]
Choose a torsion-free finite-index subgroup $\Gamma_0<\Gamma$ consisting of orientation-preserving isometries. Then $\Gamma_0\backslash\HH^3$ is a closed hyperbolic $3$-manifold. By Kahn--Markovic \cite[Theorem~1.1]{KM}, $\Gamma_0$ contains a closed quasi-Fuchsian surface subgroup $H$. It is convex cocompact, and its limit set is a Jordan curve. Replacing $H$ by its subgroup of index at most two preserving the two complementary domains gives a convex cocompact codimension-one subgroup of $\Gamma$. Theorem~\ref{thm:main} applies to $\Gamma$.
\end{proof}

\begin{proof}[Proof of Corollary~\ref{cor:geodesic}]
The totally geodesic immersion has compact domain, since it is proper and $\HH^n/\Gamma$ is compact. The image of a lift of the immersion to the universal covers is a hyperplane $\Pi\cong\HH^{n-1}$, and the corresponding subgroup $H<\Gamma$ acts cocompactly on $\Pi$. Thus $H$ is convex cocompact with $\Lambda_H=\partial \Pi$. Its subgroup of index at most two preserving both sides of $\Pi$ preserves the two components of $\partial\HH^n\setminus\partial \Pi$. Theorem~\ref{thm:main} applies.
\end{proof}

\begin{proof}[Proof of Corollary~\ref{cor:cubulation}]
By Fioravanti--Hagen \cite[Proposition~2.29]{FH21}, we may choose an essential proper cocompact action of $\Gamma$ on a $\mathrm{CAT}(0)$ cube complex: both sides of every hyperplane contain points arbitrarily far from it. Let $\Pi$ be a hyperplane and let $H$ be the subgroup of its stabilizer preserving both sides. The stabilizer acts cocompactly on the convex subspace $\Pi$, and $H$ has index at most two in it. Hence $H$ is quasiconvex in $\Gamma$, and therefore convex cocompact in $\HH^n$. The limit sets of the two halfspaces, with $\Lambda_H$ removed, give a nonempty open $H$-invariant partition of $\partial\Gamma\setminus\Lambda_H$; see \cite[Lemmas~2.3 and~2.18]{FH21}. Theorem~\ref{thm:main} applies.
\end{proof}

\bibliographystyle{siam}
\bibliography{biblio}

\end{document}